\documentclass[12pt]{article}
\usepackage[utf8]{inputenc}
\usepackage{graphicx} 
\usepackage{amssymb,amsmath,amsthm,amsfonts}
\usepackage{multicol,multirow}
\usepackage{calc}
\usepackage{ifthen}
\usepackage{mathrsfs}
\usepackage{abstract}

\usepackage{bbm}
\makeatletter
\newtheorem*{rep@theorem}{\rep@title}
\newcommand{\newreptheorem}[2]{%

\newenvironment{rep#1}[1]{%
 \def\rep@title{#2 \ref{##1}}%
 \begin{rep@theorem}}%
 {\end{rep@theorem}}}
\makeatother
\newtheorem{theorem}{Theorem}[section]
\newtheorem{proposition}[theorem]{Proposition}
\newtheorem{lemma}[theorem]{Lemma}

\theoremstyle{definition}
\newtheorem{definition}[theorem]{Definition}

\newtheorem*{remark}{Remark}
\usepackage[toc,page]{appendix}
\usepackage[a4paper, total={6.5in, 8.5in}]{geometry}
\usepackage{enumitem}
\usepackage{titlesec}

\titleformat{\section}
  {\normalfont\small\centering\scshape}{\thesection}{1em}{}

\titleformat{\subsection}
  {\normalfont\small\centering\scshape}{\thesection}{1em}{}

\newreptheorem{theorem}{Theorem}

\newcommand{\dd}{\mathrm{d}}
\newcommand{\Mod}[1]{\ \mathrm{mod}\ #1}

\DeclareMathOperator{\Var}{Var}

\makeatletter \renewcommand{\maketitle}{ \begin{center} {\normalsize\bfseries\MakeUppercase \@title \par} \vskip 1em {\small\MakeUppercase \@author \par} \vskip 1em {\small \@date \par} \end{center} } \makeatother
\renewenvironment{abstract}{ \noindent\text{\small A\scshape bstract. } \noindent\ignorespaces }{ \par }

\title{Short Sums of the Liouville Function over Function Fields}
\date{}
\author{Simon Fleet}

\begin{document}
\maketitle

\begin{abstract}
    \noindent
    Let $\lambda$ denote the Liouville function for function fields. We prove that for a fixed $q$, given $h \ll \sqrt{N}$ and $h(N) \to \infty$ arbitrarily slowly as $N \to \infty$, then
    \begin{equation*}
        \frac{1}{q^N}\sum_{G_0 \in \mathcal{M}_N}|\sum_{G \in \mathcal{I}_{h}(G_0)}\lambda(G)|^2 \ll_q \frac{N^5}{h^2}q^{h}.
    \end{equation*}
    The proof follows a similar method of an analogous case in the integer setting developed by Chinis, adapting methods originally developed by Matomäki and Radziwiłł.
\end{abstract}

\section{Introduction}
This paper studies how the Liouville Lambda function for function fields behaves in short intervals. Over the integers, we define $\lambda: \mathbb{N} \to \mathbb{C}$, to be completely multiplicative, taking the value $\lambda(p)=-1$ for all primes $p$. Since it is not expected that the integers favour having an even (or odd) number of prime factors we expect some cancellation in its partial sums. It is well known that showing \[\sum_{n \leq x}\lambda(n) = o(x),\]
is equivalent to the Prime Number Theorem (see \cite[Theorem~4.14]{apostol2013introduction}). Furthermore, if $\lambda$ can be modelled by an independent random variable, taking value $\pm 1$ with probability $1/2$, then, for any $\epsilon > 0$, we expect 
\begin{equation*}
        \sum_{n \leq x}\lambda(n) \ll_\epsilon x^{1/2 + \epsilon},
\end{equation*}
as $x \to \infty$. Proving this turns out to be equivalent to the error term given by the Riemann Hypothesis \cite[Theorem~14.25C]{TRZF}.\\

Over function fields \(\lambda : \mathbb{F}_q[t] \to \mathbb{C}\) is similarly defined to taking the value \(\lambda(P) = -1\) for every prime polynomial \(P\). Taking sums of $\lambda$ over \(\mathcal{M}_n\), it is not too hard to show that
\[\sum_{G \in \mathcal{M}_n}\lambda(G) = (-1)^n q^{\lceil \frac{n}{2} \rceil}.\]

Since \(|\mathcal{M}_n|=q^n\), this shows that in the function fields setting $\lambda$ exhibits square root cancellation over \(\mathcal{M}_n\). It is also true that cancellation occurs in shorter intervals. To determine how much cancellation occurs, we calculate the average size of sums of \(\lambda\) over short intervals.\\ 

To study short intervals over function fields, we typically have two regimes. Either we let the size of the finite field, $q \to \infty$ and fix $N$, the degree of polynomials we are summing over. Or we fix $q$ and let $N \to \infty$. Progress towards the large $q$ case has been made by Keating and Rudnick \cite{KRvar}, proving the following:

\begin{theorem}\cite[Theorem~1.2]{KRvar}
    \label{largeqvar}
     If $0 \leq h \leq N-5$, then as $q \to \infty$, $q$ odd, 
    \begin{equation*}
        \frac{1}{q^N}\sum_{G_0 \in \mathcal{M}_N}|\sum_{G \in \mathcal{I}_h(G_0)}\mu(G)|^2 \sim q^{h+1}.
    \end{equation*}
\end{theorem}
Since $|\mathcal{I}_h(G_0)| = q^{h+1}$, a trivial bound for this is $q^{2(h+1)}$, so this theorem shows square root cancellation in sums of $\mu$ on average, across all short intervals. Whilst this calculates the variance for $\mu$, it is closely linked to $\lambda$ so the corresponding result will be close in size. One of the main methods used in the proof, of the large $q$ setting, is an equidistribution result developed by Katz \cite{Katzequi}. This result only holds when $q \to \infty$, hence other tools will be required for the fixed \(q\) case.\\

Progress has been made, by Chinis \cite{CHsuml}, towards calculating the variance of $\lambda$ in short intervals over the integers, proving the following conditional result.
\begin{theorem}\cite[Theorem~1.2]{CHsuml} 
    \label{Chinisres}
    Assuming the Riemann hypothesis
    \begin{equation*}
        \int_X^{2X}|\sum_{x \leq n \leq x+h}\lambda(n)|^2 dx \ll Xh(\log X)^6,
    \end{equation*}
    as $X \to \infty$ provided $h = h(X) \leq \exp{\left(\sqrt{(\frac{1}{2}-o(1))\log X \log\log X}\right)}$.
\end{theorem}

\begin{remark}
    This shows, for $(\log X)^{6+\epsilon} \leq h \leq \exp{\left(\sqrt{(\frac{1}{2}-o(1))\log X \log\log X}\right)}$, that $\lambda$ exhibits square root cancellation in almost all short intervals. 
\end{remark}

The idea behind this proof is to vary the proof of the Matomäki-Radziwiłł theorem \cite{MRthm}. Short sums are converted into longer sums in the Fourier space. Here we need to understand the average value of Dirichlet polynomials. The key tool for evaluating these sums is the $L^2$ mean value theorem given below.

\begin{lemma}
    For any sequence of complex numbers $\{a_n\} \subset \mathbb{C}$ we have:
    \begin{equation*}
        \int_0^T |\sum_{1 \leq n \leq N} a_n n^{it}|^2 dt 
        = (T + O(N))\sum_{1 \leq n \leq N}|a_n|^2.
    \end{equation*}
\end{lemma}
\begin{proof}
    See \cite[Theorem~9]{ANT}.
\end{proof}

When $N \ll T$ the diagonal terms become dominant, causing this lemma to become particularly strong. In the proof of Theorem \ref{Chinisres}, they begin with $N=X$ and $T \simeq X/h$. The main idea is to split the sum in such a way that $N \ll T$, then applying the mean value theorem gives strong bounds.\\

In a recent paper by Klurman, Mangerel, and Teräväinen \cite{corrff} an analogue of the Matomäki-Radziwiłł Theorem for function fields is proven. They develop numerous tools, including an analogue of the $L^2$ mean value theorem. See \cite{klurman2023multiplicative} for related results and techniques over number fields. This allows us to adapt the method of proof used by Chinis (Theorem \ref{Chinisres}) to the function fields setting, proving the following:

\begin{theorem}\label{Result}
    For a fixed $q$. Given $h \ll \sqrt{N}$ and $h(N) \to \infty$ arbitrarily slowly as $N \to \infty$. Then
    \begin{equation*}
        \frac{1}{q^N}\sum_{G_0 \in \mathcal{M}_N}|\sum_{G \in \mathcal{I}_{h}(G_0)}\lambda(G)|^2 \ll_q \frac{N^5}{h^2}q^{h}.
    \end{equation*}
\end{theorem}

As $h$ grows, $q^h$ becomes the dominant term in the bound. For $h = a \log_q N$ the trivial bound is $q^{2a \log_q N}$, thus Theorem \ref{Result} gives a saving of $q^{a-5}$. Therefore, when $h > a\log_q N$ with $a$ much larger than $5$, Theorem \ref{Result} shows square root cancellation on average up to $\log$ factors, provided $h \ll \sqrt{N}$. Equating the short interval $\mathcal{I}_h(G_0)$ to the interval $[x, x+h]$ in the integer setting, this result shows square root cancellation on average for $h > (\log X)^{5 + \epsilon}$, so we see this is analogous to Theorem \ref{Chinisres}.\\

As with the integer setting, we obtain a smaller range of values that $h$ is allowed to take, however, the corresponding bound of the variance is much stronger than that given by the Matomäki-Radziwiłł Theorem. Comparing this with Theorem \ref{largeqvar} we have the extra $\frac{N^5}{h^2}$ term. In the large $q$ scenario, $N$ is fixed and thus $\frac{N^5}{h^2}$ is a constant term. Therefore, Theorem \ref{Result} equates to a bound of similar size. Since, $q \to \infty$, and $N$ is finite, each multiple of $q$ corresponds to a large change in the variance. Whereas, in the fixed $q$ case, each multiple of $q$ corresponds to only a small change in size.\\

\section*{Acknowledgements}
The author would like to thank Jake Chinis for suggesting this problem and his continuous support and encouragement with the research in this paper. The author is also grateful to Oleksiy Klurman for providing insightful discussions and comments that have helped improve aspects of this paper.

\pagebreak

\section{Notation}
Let $p$ be the characteristic of $\mathbb{F}_q[t]$, where $q = p^k$ for some $k \geq 1$. $\mathbb{F}_q[t]$ is the ring of polynomials over $\mathbb{F}_q$. 

$\mathcal{M}$ is defined to be the set of monic polynomials over $\mathbb{F}_q$ and $\mathcal{P}$ is the set of monic irreducible (prime) elements. For $n \in \mathbb{N}$, let $\mathcal{M}_n$, $\mathcal{M}_{<n}$ and $\mathcal{M}_{\leq n}$ be the subset of $\mathcal{M}$ consisting of polynomials of degree $n$, $<n$ and $\leq n$ respectively. With $\mathcal{P}_n$, $\mathcal{P}_{< n}$, $\mathcal{P}_{\leq n}$ defined similarly. The degree of a polynomial $F \in \mathbb{F}_q[t]$ is denoted $\deg(F)$. We say that $|G| = q^{\deg(G)}$.

The short interval of length $h$ centred about $G_0$ is given by
\[\mathcal{I}_{h}(G_{0}) = \{G \in \mathbb{F}_q[t] : \deg(G - G_{0}) \leq h\}.\]

The greatest common denominator of $F$ and $G$ is $H$, where $H$ is the polynomial of the largest degree such that $H|F$ and $H|G$, denoted $(F, G) = H$. $F$ and $G$ are said to be coprime if $(F,G) = 1$. A function $f: \mathcal{M} \to \mathbb{C}$ is said to be multiplicative if $f(FG) = f(F)f(G)$ whenever $(F,G)=1$. Further, we say that $f$ is completely multiplicative if we can remove the condition that $F$ and $G$ must be coprime.\\
\begin{itemize}[noitemsep]
    \item $\omega: \mathcal{M} \to \mathbb{C}$ counts the number of distinct irreducible factors of a given element of $\mathcal{M}$.
    \item $\mu: \mathcal{M} \to \mathbb{C}$ is given by $\mu(F) = (-1)^{\omega(F)}$ if $F$ is square-free, else $\mu(F) = 0$.
    \item $\lambda: \mathcal{M} \to \mathbb{C}$ is completely multiplicative, with $\lambda(P) = - 1$ for every $P \in \mathcal{P}$.
    \item $\phi: \mathcal{M} \to \mathbb{C}$ is given by $\phi(F) = |\mathbb{F}_q[f]/F\mathbb{F}_q[t]^*|$. 
    \item $\Lambda: \mathcal{M} \to \mathbb{C}$ is given by $\Lambda(F) = \deg(P)$ if $F = P^k$, else $0$.
    \item $\pi_q:\mathbb{N} \to \mathbb{C}$, is given by \(\pi_q(n) = |\mathcal{P}_n|\). This counts the number of prime polynomials of a given degree.
\end{itemize}

 We say a polynomial $F \in \mathbb{F}_q[t]$, is \emph{$h$-smooth} if all of its irreducible factors are of degree at most $h$. Let $\mathcal{S}_{h,n}$ denote the set of \emph{$h$-smooth} polynomials of degree $n$.
 A function $f : \mathbb{F}_q[t] \to \mathbb{C}$ is said to be \emph{even} if for all $F \in \mathbb{F}_q[t]$ we have $f(cF) = f(F)$, for any $c \in \mathbb{F}_q^*$.
A Dirichlet character of modulus $H$ is a function $\chi: \mathbb{F}_q[t] \to \mathbb{C}$, if for all $F, G \in \mathbb{F}_q[t]$ we have:
    \begin{equation*}
        \chi(FG) = \chi(F)\chi(G)
    \end{equation*}
    \begin{equation*}
        \chi(F)  
        \begin{cases}
            =0, & gcd(F,H) > 1\\
            \neq 0, & gcd(F,H) =1\\
        \end{cases}
    \end{equation*}
    \begin{equation*}
        \chi(F+H) = \chi(F)
    \end{equation*}

Define $\Phi(Q)$ to be the number of Dirichlet characters modulo $Q$. Also, define $\Phi_{ev}(Q)$ to be the number of even characters modulo $Q$.

\section{Preliminary Results}
This section introduces the notion of variance in short intervals providing a starting place for the proof of the Theorem \ref{Result}. Other results are introduced, including the $L^2$ mean value theorem for function fields and an involution for polynomials.
\begin{definition}
    The variance of a function $f$, in short intervals of length $h$ centred around degree $N$ polynomials, is given by
    \[\Var(f_{N,h}) = \frac{1}{q^N}\sum_{G_0 \in \mathcal{M}_N}|\sum_{G \in \mathcal{I}_h(G_0)}f(G)|^2.\]
\end{definition}

Using methods developed by Keating and Rudnick \cite{KRvar}, the formula for the variance is converted into weighted character sums over $\mathcal{M}_n$ that are easier to work with. This is done by converting short intervals into arithmetic progressions via an involution of polynomials. From here orthogonality relations for characters are used, allowing us to sum over $\mathcal{M}_n$, where only the polynomials in the arithmetic progression contribute a non-zero value.\\

We begin by introducing an involution function which shall be used to convert short intervals into arithmetic progressions.

\begin{definition}\cite{KRvar}
    The map $*: \mathbb{F}_q[t] \to \mathbb{F}_q[t]$ is defined by
    \[F(t)^* = t^{\deg(F)} F(1/t).\]
    $f:\mathbb{F}_q[t] \to \mathbb{C}$ is \emph{symmetric} if $f(F) = f(F^*)$ whenever $F$ is coprime to $t$.
\end{definition}

\begin{lemma}\cite{KRvar}
If $F$ satisfies $\gcd(F,t) = 1$ then $*$ is self inverse. i.e. we have $F^{**} = F$.
\end{lemma}

\begin{lemma}\cite[Lemma~5.4]{KRvar} 
    \label{Variance Formula}
    If $f$ is even, symmetric and multiplicative, and $0 \leq h \leq n-2$, then
    \begin{equation*}
        \Var(f_{N,h}) = \frac{1}{\phi_{ev}(t^{N-h})^2}\sum_{\substack{{\chi \Mod{t^{n-h}}}\\{\chi_{even}}}}|\sum_{n=0}^{N}f(t^n)\sum_{G \in \mathcal{M}_n}f(G)\chi(G)|^2.
    \end{equation*}
\end{lemma}
By definition, $\lambda$ is even and multiplicative, all that is left to show is that it is symmetric.

\begin{lemma}\label{symmetry of Liouville}
    $\lambda$ is symmetric. Explicitly
    \(\lambda(F) = \lambda(F^*)\), whenever $(F,t)=1$.
\end{lemma}

\begin{proof}
    Let $F$ be coprime to $t$, with $F = GP$ for some prime $P$. By multiplicity of $*$ and $\lambda$, we have
    \[\lambda(F)^* = \lambda(GP)^* = \lambda(G^*)\lambda(P^*).\]
    It remains to show that \(\lambda(P) = \lambda(P^*)\) for all primes $P$.
    Since $\lambda(P) = - 1$ by definition, we can show that the primes are closed by $*$. Suppose $F = GH$ for some non constant factors $G, H$, then \[F^* = G^* H^*.\] So if $F$ is composite then so is $F^*$. Since $*$ is a bijection, this implies the primes are closed under $*$. Therefore \(\lambda(P) = \lambda(P^*)\) for all primes $P$, and the result follows by multiplicity of $\lambda$.
    
\end{proof}

The following is the analogue of the $L^2$ mean value theorem for function fields, it shall be the main tool used in proving Theorem \ref{Result}.
\begin{lemma}\cite[Lemma~4.2]{corrff}
\label{MVT}
For $n \geq 1$. Let $\{a_G\}_{G \in \mathcal{M}_n} \subset \mathbb{C}$ and $Q \in \mathcal{M}$. Then
    \begin{equation*}
        \sum_{\chi \Mod{Q}}|\sum_{G \in \mathcal{M}_{n}}a_{G} \chi(G)|^{2} \leq 2\phi(Q)(q^{n-deg(Q)} + 1) \sum_{\substack{{G \in \mathcal{M}_{n}}\\{(G,Q) =1}}}|a_{G}|^{2}.
    \end{equation*}
\end{lemma}

\begin{remark}
    To make good use of the mean value theorem, we observe there are two main terms: 
    \begin{itemize}[noitemsep] 
    \item The first being \( q^{n-\deg(Q)} + 1 \). Thus, applying it to sums over \( \mathcal{M}_n \) where \( n \) is small compared to \( \deg(Q) \) should give us a strong bound. 
    \item The other main term in our bound is the sum of \( |a_G|^2 \). Therefore, another way to succeed with the mean value theorem is to make it so that \( a_G \) regularly takes the value zero. \end{itemize}
\end{remark}

The following lemma is to be used in the proof of our bound on a character sum, over prime polynomials of a given degree.

\begin{lemma}\cite[Theorem~3]{GeorgesRhin1972}.
    \label{weighted von mandgolt}
    Let $N \geq 1$ let $\chi$ be a non-principal character modulo $Q$. Then
    \begin{equation*}
        \sum_{G \in M_N}\Lambda(G)\chi(G) \ll \deg(Q) q^{\frac{N}{2}}.
    \end{equation*}
\end{lemma}
 
\begin{proposition}\label{l1 pcs}
    Let $\chi$ be a non-principal character modulo $t^{N-h}$ then
    \begin{equation*}
        |\sum_{P \in P_x}\chi(P)| \ll \frac{N-h}{x}q^{\frac{x}{2}}.
    \end{equation*}
\end{proposition}
\begin{proof}
    By the Prime Polynomial Theorem:
    \begin{equation*}
        \sum_{\substack{{P^k \in M_x}\\{k \geq 2}}}\Lambda(P^k)\chi(P^k) =  O(q^{\frac{x}{2}}).
    \end{equation*}
    Thus
    \begin{equation*}
        \begin{split}
            |\sum_{G \in M_x}\Lambda(G)\chi(G)| &= |\sum_{P \in P_x}x \chi(P) + \sum_{\substack{{P^k \in P_x}\\{k \geq 2}}}\Lambda(P^k)\chi(P^k)|\\
            &= x|\sum_{P \in P_x}\chi(P)| + O(q^\frac{x}{2})\\
            &\ll (N-h)q^{\frac{x}{2}},
        \end{split}
    \end{equation*}
    and the result follows.
\end{proof}

The following is an analogue of Ramaré's identity which will give us a method of extracting sums of the form 
\(\sum_{P \in \mathcal{P}_x}\chi(P)\)
out of our weighted character sum in the variance formula.

\begin{lemma}\cite[Lemma~4.14]{corrff}
\label{Ramaréff} Let $n > h \leq 1$. Let $f:M\to \mathbb{U}$ be multiplicative. Then for any $G$ with an irreducible factor $R$ satisfying $\deg(R) \in [h,n]$, we have
    \begin{equation*}
        f(G) = \sum_{\substack{{RM=G}\\{R \in \mathcal{P}}\\{deg(R)\in [h,n]}}}\frac{f(RM)}{1_{(R,M)=1} + \omega_{[h,n]}(M)}.
    \end{equation*}
    Where $\omega_{[h,n]}(M) := |\{R \in \mathcal{P}, \deg(R) \in [h,n], R|M\}|$.
\end{lemma}

Recall that $\mathcal{S}_{h, N}$ is the set of $h$-smooth monic polynomials of degree $N$. The following gives us a bound on the size of $\mathcal{S}_{h, N}$.
\begin{theorem}\cite[Theorem~5.2]{Nohsmooth}
\label{No smooth}
    Let  $h \ll N$, then we have
    \begin{equation*}
        |\mathcal{S}_{h,N}| = q^N e^{-(1 + o(1))\frac{N}{h}\log(\frac{N}{h})}.
    \end{equation*}
\end{theorem}

The final result is Minkowski's inequality for integrals, it shall be used to exchange the order of summations in the variance formula, enabling appropriate use of the mean value theorem.
\begin{lemma}
    Let $F(x,y)$ be a measurable function on the $\sigma$-finite product measure spaces: $(S_1, \dd x)$ and $(S_2, \dd y)$. Then
    \label{MII}
    \begin{equation*}
        \left(\int_{S_1} \left( \int_{S_2}\left| F(x,y)\right| \dd x \right) ^p \dd y  \right)^{1/p}
        \leq \int_{S_2}\left( \int_{S_1} |F(x,y)|^p \dd y \right)^{1/p} \dd y.
    \end{equation*}
\end{lemma}
\begin{proof}
    See \cite[Theorem~A1]{stein1970singular}.
\end{proof}

The tools required to prove Theorem \ref{Result} have now been developed, the final section is dedicated to the proof of this result.

\section{Proof of Theorem \ref{Result}}
 In this section, we prove our main result, restated for ease of reading.

\begin{reptheorem}{Result}
    For a fixed $q$. Given $h \ll \sqrt{N}$ and $h(N) \to \infty$ arbitrarily slowly as $N \to \infty$. Then
    \begin{equation*}
        \frac{1}{q^N}\sum_{G_0 \in \mathcal{M}_N}|\sum_{G \in \mathcal{I}_{h}(G_0)}\lambda(G)|^2 \ll_q \frac{N^5}{h^2}q^{h}.
    \end{equation*}
\end{reptheorem}

The starting point is Lemma \ref{Variance Formula}, converting our short interval sums into weighted character sums, over long intervals.

\[\begin{split}
    \Var(\lambda_{N,h}) &= \frac{1}{\phi_{ev}(t^{N-h})^2}\sum_{\substack{{\chi \Mod{t^{N-h}}}\\{\chi \text{ even}}}}|\sum_{n=0}^{N}\lambda(t^n)\sum_{G \in \mathcal{M}_n}\lambda(G)\chi(G)|^2\\
    &\leq \frac{1}{\phi_{ev}(t^{N-h})^2}\sum_{\substack{{\chi \Mod{t^{N-h}}}\\{\chi \text{ even}}}} \left( \sum_{n=0}^{N}| \lambda(t^n)\sum_{G\in \mathcal{M}_n}\lambda(G)\chi(G) | \right)^2.
\end{split}\]

To use the $L^2$ mean value theorem appropriately, we use Minkowski's inequality for integrals to exchange the order of summation. By Lemma \ref{MII}
\[\begin{split}
    \Var(\lambda_{N,h}) 
    &\leq \frac{1}{\phi_{ev}(t^{N-h})^2}\left( \sum_{n=0}^{N}\left( \sum_{\substack{{\chi \Mod{t^{N-h}}}\\{\chi \text{ even}}}} |\lambda(t^n) \sum_{G \in \mathcal{M}_n}\lambda(G)\chi(G)|^2 \right)^{1/2} \right)^2\\
    &= \frac{1}{\phi_{ev}(t^{N-h})^2}\left( \sum_{n=0}^{N}\left( \sum_{\substack{{\chi \Mod{t^{N-h}}}\\{\chi \text{ even}}}} |\sum_{G \in \mathcal{M}_n}\lambda(G)\chi(G)|^2 \right)^{1/2} \right)^2.
\end{split}\]

The plan for the proof is to follow methods used by Chinis \cite{CHsuml} to shorten the sum, allowing us to take advantage of Lemma \ref{MVT}. The inner sum is split into polynomials with and without a prime factor larger than $h$. If $h$ is small enough; the sum over polynomials without a large prime factor will be short enough to simply use Theorem \ref{MVT} to obtain a strong bound. For the polynomials with a large irreducible factor, we will extract large prime factors of certain degrees, leaving us with a sum over $\mathcal{M}_{n-x}$, where $x \geq h$. Making the sum of appropriate length to use Lemma \ref{MVT}.\\

\begin{equation} \label{split fml}
    \begin{split}
        \sum_{\substack{{\chi \Mod{t^{N-h}}}\\{\chi \text{ even}}}}|\sum_{G \in \mathcal{M}_n}\lambda(G)\chi(G)|^2
        &\ll  \sum_{\substack{{\chi \Mod{t^{N-h}}}\\{\chi \text{ even}}}}|\sum_{\substack{{G \in \mathcal{M}_n}\\{G \in S(h,n)}}}\lambda(G)\chi(G)|^{2}\\ 
        &+  \sum_{\substack{{\chi \Mod{t^{N-h}}}\\{\chi \text{ even}}}}|\sum_{\substack{{G \in \mathcal{M}_{n}}\\{G \notin S(h,n)}}}\lambda(G)\chi(G)|^{2}.
    \end{split}
\end{equation}

Theorem \ref{Result} will be proven if for $h(N) \ll \sqrt{N}$ both sums are of size $\ll N^3 q^h$. The next section addresses the polynomials with at least one large prime factor.\\

\subsection*{Polynomials with a Large Prime Factor}
When $h$ is small, the contribution from the polynomials with a large prime factor will be larger. In this section, we calculate the sum for all $h$, evaluating how small $h$ can be such that the contribution is still small enough.

\begin{proposition}\label{LargePF}
    Given that $h(N) \to \infty$ as $N \to \infty$, we have
    \begin{equation*}
        \sum_{\substack{{\chi \Mod{t^{N-h}}}\\{\chi \text{ even}}}}|\sum_{\substack{{G \in M_{n}}\\{G \notin S(h,n)}}}\lambda(G)\chi(G)|^2 \ll (n-h)\left(\frac{N}{h}\right)^2 q^{N+n-h}.
    \end{equation*}
\end{proposition}

We first note when $h$ is small, the vast majority of the polynomials fall into this category. However, when $h \ll \log N$ the bound becomes very weak and no longer leads to square root cancellation. For small values of $h$, we obtain a bound of size roughly $N^3 q^{2N}$, whereas a trivial bound is $q^{3N}$, thus we still see some cancellation; however, it is not square-root. Thus, this method only gives the required cancellation for the variance once $h \gg \log N$.\\

Following Chinis's proof of \cite[Proposition~4.1]{CHsuml}, we start with a lemma that splits the sum. The first, larger sum, allows us to use our bound on $\sum_{P \in \mathcal{P}_x} \chi(P)$. The second sum is small enough that we can make good use of Lemma \ref{MVT}.

\begin{lemma}
    \begin{equation*}
    \begin{split}
        \sum_{\substack{{\chi \Mod{t^{N-h}}}\\{\chi \text{ even}}}}|\sum_{\substack{{G \in \mathcal{M}_{n}}\\{G \notin S(h,n)}}}\lambda(G)\chi(G)|^2
        &\ll \sum_{\substack{{\chi \Mod{t^{N-h}}}\\{\chi \text{ even}}}}\sum_{h < x \leq n}|\sum_{P \in \mathcal{P}_x}\chi(P)|^2 |\sum_{M \in \mathcal{M}_{n-x}}a_M \chi(M)|^2\\
        &+ \sum_{\substack{{\chi \Mod{t^{N-h}}}\\{\chi \text{ even}}}}\sum_{h < x \leq n}|\sum_{P \in \mathcal{P}_x}\chi(P)^2 \sum_{M \in \mathcal{M}_{n-2x}}b_{MP} \chi(M)|^2,
    \end{split}
    \end{equation*}
    where $a_M$ and $b_{MP}$ are to be determined, satisfying $|a_M|, |b_{MP}| \leq 1$.
\end{lemma}

\begin{proof}
We begin with the analogue of Ramaré's identity to create sums over prime polynomials of a large degree. Using Lemma \ref{Ramaréff} and rearranging the summations, we have:

\begin{equation*}
        \begin{split}
            \sum_{\substack{{G \in M_{n}}\\{G \notin S(h,n)}}}\lambda(G)\chi(G) 
            &= \sum_{\substack{{G \in \mathcal{M}_{n}}\\{G \notin S(h,n)}}}\sum_{\substack{{PM = G}\\{P \in \mathcal{P}}\\{deg(P) \in [h,n]}}}\frac{\lambda(P)\chi(P)\lambda(M)\chi(M)}{\omega_{(h,n)}(M) + 1_{(P,M)=1}}\\ 
            &= \sum_{h < x \leq n}\sum_{P \in \mathcal{P}_x}\sum_{M \in \mathcal{M}_{n-x}} \frac{\lambda(P)\chi(P)\lambda(M)\chi(M)}{\omega_{(h,n)}(M) + 1_{(P,M)=1}}\\
            &= \sum_{h < x \leq n}\sum_{P \in \mathcal{P}_{x}} \lambda(P)\chi(P) \sum_{M \in \mathcal{M}_{n-x}}\frac{\lambda(M)\chi(M)}{\omega_{(h,n)}(M) + 1_{(P,M)=1}}.
        \end{split}
    \end{equation*}

We have extracted $\sum_{P\in \mathcal{P}_x}\lambda(P)\chi(P)$, however, the inner sum is still dependent on $1_{(P, M)=1}$. To complete the proof, we shall split the inner sum into a large sum, without dependence on $P$, and a smaller sum with dependence on $P$. Noting $\lambda(P)=-1$, we have:
\begin{equation*}
    \begin{split}
        & \sum_{P \in \mathcal{P}_{x}} \lambda(P)\chi(P) \sum_{M \in \mathcal{M}_{n-x}}\frac{\lambda(M)\chi(M)}{\omega_{(h,n)}(M) + 1_{(P,M)=1}}\\
        &= \sum_{P \in \mathcal{P}_{x}}\chi(P) \sum_{\substack{M \in \mathcal{M}_{n-x}\\{(P,M) = 1}}}\frac{-\lambda(M)\chi(M)}{\omega_{(h,n)}(M) + 1}\\ 
        &+ \sum_{P \in \mathcal{P}_{x}}\chi(P) \sum_{\substack{M \in \mathcal{M}_{n-x}\\{(P,M) \neq 1}}}\frac{-\lambda(M)\chi(M)}{\omega_{(h,n)}(M)}.\\
    \end{split}
\end{equation*}

To remove the dependence, of $P$, in the first term, we add $M \in \mathcal{M}_{n-x}$ s.t. $P|M$ to the first term and subtract it from the second.

\begin{equation*}
\begin{split}
    & \sum_{\substack{{G \in \mathcal{M}_n}\\{G \notin S_{h,n}}}}\lambda(G)\chi(G) = \sum_{P \in \mathcal{P}_{x}}\chi(P) \sum_{\substack{M \in \mathcal{M}_{n-x}}}\frac{-\lambda(M)\chi(M)}{\omega_{(h,n)}(M) + 1}\\
    &+ \sum_{P \in \mathcal{P}_{x}}\chi(P)\sum_{\substack{{M \in \mathcal{M}_{n-x}}\\{(P,M)\neq 1}}}\chi(M)\left( \frac{\lambda(M)}{\omega_{(h,n)}(M)+1} + \frac{-\lambda(M)}{\omega_{(h,n)}(M)} \right)\\
    &= \sum_{P \in \mathcal{P}_{x}}\chi(P) \sum_{\substack{M \in \mathcal{M}_{n-x}}}\frac{-\lambda(M)\chi(M)}{\omega_{(h,n)}(M) + 1} \\
    &+ \sum_{P \in \mathcal{P}_{x}}\chi(P)\sum_{\substack{{M \in \mathcal{M}_{n-x}}\\{(P,M)\neq 1}}} \frac{-\lambda(M)\chi(M)}{\omega_{(h,n)}(M)(\omega_{(h,n)}(M)+1)}.
\end{split}
\end{equation*}

Following Chinis \cite{CHsuml}, for brevity we define;\\
$a_{M} := \frac{-\lambda(M)}{\omega_{(h,n)}(M)+1}$, $b_{M} := \frac{-\lambda(M)}{\omega_{(h,n)}(M)(\omega_{(h,n)}(M) + 1)}$, this gives:\\

\begin{equation*}
    \begin{split}
        & \sum_{\substack{{G \in \mathcal{M}_n}\\{G \in S_{h,n}}}}\lambda(G)\chi(G) = \sum_{h < x \leq n}[\sum_{P \in \mathcal{P}_x}\chi(P)\sum_{M \in \mathcal{M}_{n-x}}a_m \chi(m) + \sum_{P \in \mathcal{P}_x}\chi(P)\sum_{\substack{{M \in \mathcal{M}_{n-x}}\\{(P,M)\neq 1}}}b_M \chi(M)]\\
        &= \sum_{h < x \leq n}[\sum_{P \in \mathcal{P}_x}\chi(P)\sum_{M \in \mathcal{M}_{n-x}}a_M \chi(M) + \sum_{P \in \mathcal{P}_x}\chi(P)^2 \sum_{M \in \mathcal{M}_{n-2x}}b_{MP} \chi(M)].
    \end{split}
\end{equation*}

Where in the last step a factor of $P$ has been removed from our sum over polynomials $M \in \mathcal{M}_{n-x}$ such that $P|M$. Finally, this decomposition is substituted into our desired sum where we then use the triangle inequality.

\begin{align}
    &\sum_{\substack{{\chi \Mod{t^{N-h}}}\\{\chi \text{ even}}}}|\sum_{\substack{{G \in \mathcal{M}_{n}}\\{G \in \mathcal{S}_{h,n}}}}\lambda(G)\chi(G)|^{2}\nonumber\\
    &= \sum_{\substack{{\chi \Mod{t^{N-h}}}\\{\chi \text{ even}}}}|\sum_{h < x \leq n}[\sum_{P \in \mathcal{P}_x}\chi(P)\sum_{M \in \mathcal{M}_{n-x}}a_M \chi(M) + \sum_{P \in \mathcal{P}_x}\chi(P)^2 \sum_{M \in \mathcal{M}_{n-2x}}b_{MP} \chi(M)]|^2\nonumber\\
    &\ll \sum_{\substack{{\chi \Mod{t^{N-h}}}\\{\chi \text{ even}}}}\sum_{h < x \leq n}|\sum_{P \in \mathcal{P}_x}\chi(P)\sum_{M \in \mathcal{M}_{n-x}}a_M \chi(M)|^2\nonumber\\
    &+ \sum_{\substack{{\chi \Mod{t^{N-h}}}\\{\chi \text{ even}}}}\sum_{h < x \leq n}|\sum_{P \in \mathcal{P}_x} \sum_{M \in \mathcal{M}_{n-2x}}\chi(P)^2 b_{MP} \chi(M)|^2\nonumber\\
    &= \sum_{\substack{{\chi \Mod{t^{N-h}}}\\{\chi \text{ even}}}}\sum_{h < x \leq n}|\sum_{P \in \mathcal{P}_x}\chi(P)|^2|\sum_{M \in \mathcal{M}_{n-x}}a_M \chi(M)|^2\nonumber\\
    &+ \sum_{\substack{{\chi \Mod{t^{N-h}}}\\{\chi \text{ even}}}}\sum_{h < x \leq n}|\sum_{P \in \mathcal{P}_x} \sum_{M \in \mathcal{M}_{n-2x}}\chi(P)^2 b_{MP} \chi(M)|^2\nonumber
\end{align}

giving the desired result.
\end{proof}

Now we can evaluate the sums above separately. We start with the second, smaller sum. 

\begin{lemma}\label{smaller sum}
    \begin{equation*}
        \sum_{\substack{{\chi \Mod{t^{N-h}}}\\{\chi \text{ even}}}}\sum_{h < x \leq n}|\sum_{P \in \mathcal{P}_x}\sum_{M \in \mathcal{M}_{n-2x}}\chi(P)^2 b_{MP} \chi(M)|^2 \ll \frac{n}{h^2}q^{N+n-h}.
    \end{equation*}
\end{lemma}
Now a trivial bound is of size $\frac{n}{h^2}q^{N+2n - 3h}$, therefore this better by a factor of $q^{n-2h}$, assuming $n$ is significantly larger than $h$. If $n$ is small then the trivial bound is already strong.

\begin{proof}
\begin{equation*}
\begin{split}
    &\sum_{\substack{{\chi \Mod{t^{N-h}}}\\{\chi \text{ even}}}}\sum_{h < x \leq n}|\sum_{P \in \mathcal{P}_x}\sum_{M \in \mathcal{M}_{n-2x}}\chi(P)^2 b_{MP}\chi(M)|^2\\
    &= \sum_{h < x \leq n} \sum_{\substack{{\chi \Mod{t^{N-h}}}\\{\chi \text{ even}}}}|\sum_{M \in \mathcal{M}_{n-2x}}\sum_{P\in \mathcal{P}_x}\chi(P)^2 b_{MP} \chi(M)|^2
\end{split}
\end{equation*}
Using Lemma \ref{MVT}, with $a_G = \sum_{P\in \mathcal{P}_x}\chi(P)^2 b_{MP} \chi(M)$ this is
\begin{equation*}
\begin{split}
    &\ll \sum_{h < x \leq n}q^{N-h}(q^{n-2x}+1)\sum_{M\in \mathcal{M}_{n-2x}}|\sum_{P\in \mathcal{P}_x}\chi(P)^2 b_{MP}\chi(M)|^2\\
    &\ll \sum_{h < x \leq n}q^{N-h} \sum_{M \in \mathcal{M}_{n-2x}}|\mathcal{P}_x|^2\\
    &\leq \sum_{h < x \leq n}q^{N-h} q^{n-2x} \frac{q^{2x}}{x^2}\\
    &\ll \frac{n}{h^2}q^{N+n-h}.
\end{split}
\end{equation*}
\end{proof}

The sum is trivially bounded by $\frac{n}{h^2}q^{N + 2n -2h}$, by using Lemma \ref{MVT} we are saving a factor of $q^{n-h}$.
Using Lemma \ref{MVT}, our sum becomes small enough that the diagonal term is dominant, therefore we achieve the best bound possible from this lemma. To achieve stronger results, a different tool to the mean value theorem would be required. Moving onto the first larger sum, we prove the following.

\begin{lemma}\label{larger sum}
    \begin{equation*}
        \sum_{\substack{{\chi \Mod{t^{N-h}}}\\{\chi \text{ even}}}}\sum_{h < x \leq n}|\sum_{P \in \mathcal{P}_x}\chi(P)|^2|\sum_{M \in \mathcal{M}_{n-x}}a_M \chi(M)|^2 
        \ll (n-h)\frac{(N-h)^2}{h^2} q^{N+n-h}
    \end{equation*}
\end{lemma}
\begin{proof}

Using our bound on $|\sum_{P \in \mathcal{P}_x}\chi(P)|$ we have
\begin{equation*}
\begin{split}
    &\sum_{\substack{{\chi \Mod{t^{N-h}}}\\{\chi \text{ even}}}}\sum_{h < x \leq n}|\sum_{P \in \mathcal{P}_x}\chi(P)|^2|\sum_{M \in \mathcal{M}_{n-x}}a_M \chi(M)|^2\\
    &\ll \sum_{\substack{{\chi \Mod{t^{N-h}}}\\{\chi \text{ even}}}}\sum_{h < x \leq n}\left(\frac{N-h}{x}q^{\frac{x}{2}}\right)^2|\sum_{M \in \mathcal{M}_{n-x}}a_M \chi(M)|^2\\
    &= \sum_{h < x \leq n}\frac{(N-h)^2}{x^2}q^x \sum_{\substack{{\chi \Mod{t^{N-h}}}\\{\chi \text{ even}}}}|\sum_{M \in \mathcal{M}_{n-x}}a_M \chi(M)|^2
\end{split}
\end{equation*}
Now evaluating the character sum using Lemma \ref{MVT} we have

\begin{equation*}
\begin{split}
    &\sum_{\substack{{\chi \Mod{t^{N-h}}}\\{\chi \text{ even}}}}|\sum_{M \in \mathcal{M}_{n-x}}a_M \chi(M)|^2 \\
    &\ll q^{N-h}(q^{n-x-(N-h)}+1)\sum_{M \in \mathcal{M}_{n-x}}|a_M \chi(M)|^2\\
    &\ll q^{N-h}\sum_{M \in \mathcal{M}_{n-x}}1
    = q^{N+n-h-x}.
\end{split}
\end{equation*}
Where we have used that $|a_M \chi(M)| \leq 1$.\\

Substituting this in, we get
\begin{equation*}
\begin{split}
    &\sum_{h < x \leq n}\frac{(N-h)^2}{x^2}q^x \sum_{\substack{{\chi \Mod{t^{N-h}}}\\{\chi \text{ even}}}}|\sum_{M \in \mathcal{M}_{n-x}}a_M \chi(M)|^2\\
    &\ll \sum_{h < x \leq n}\frac{(N-h)^2}{x^2}q^x q^{N+n-h-x}\\
    &=\sum_{h < x \leq n}\frac{(N-h)^2}{x^2}q^{N+n-h}\\
    &\leq (n-h)\frac{(N-h)^2}{h^2}q^{N+n-h},
\end{split}
\end{equation*}
Completing the proof.
\end{proof}

\begin{proof}[Proof of Proposition \ref{LargePF}].
    Substituting Lemma \ref{smaller sum} and Lemma \ref{larger sum}:
\begin{equation*}
    \begin{split}
        \sum_{\substack{{\chi \Mod{t^{N-h}}}\\{\chi \text{ even}}}}|\sum_{\substack{{G \in \mathcal{M}_n}\\{G \notin \mathcal{S}_{h,n}}}}\lambda(G)\chi(G)|^2 &\ll \frac{N}{h^2}q^{N+n-h} + (n-h)\left(\frac{N-h}{h}\right)^2 q^{N+n-h}\\
        &\ll \frac{N^3}{h^2} q^{N+n-h},
    \end{split}
\end{equation*}

completing the proof of Proposition \ref{LargePF}.
\end{proof}

\subsection*{Smooth Polynomials}

\begin{proposition}\label{SmoothPF}
    Let $h \ll \sqrt{N}$ then as $N \to \infty$ we have
    \begin{equation*}
        \sum_{\substack{{\chi \Mod{t^{N-h}}}\\{\chi \text{ even}}}}|\sum_{\substack{{G \in M_{n}}\\{G \in \mathcal{S}_{h,n}}}}\lambda(G)\chi(G)|^{2} \ll q^{n+N-h} + q^{2(N-h)}.
    \end{equation*}
\end{proposition}

For this to hold, the smaller $h$ is the fewer number of $h$-smooth polynomials there are. To prove the proposition we bound the contribution from $h$-smooth polynomials by a function and determine $h$ such that the desired bound holds.\\

\begin{lemma}\label{bound no smooth}
    Given that $h \ll \sqrt{N}$, as $N \to \infty$, we have
    \begin{equation*}
        |\mathcal{S}_{h,N}| \ll q^{N-h}
    \end{equation*}
\end{lemma}
\begin{proof}
    As $N \to \infty$, with $h \ll N$, we have
    \begin{equation*}
    \begin{split}
        |\mathcal{S}_{h,N}| &\sim q^N e^{-\frac{N}{h}\log \frac{N}{h}}\\
        &= q^{N - \frac{N}{h} \log_q \frac{N}{h}}.
    \end{split}
    \end{equation*}
    Thus we need to show that when $h \ll \sqrt{N}$, we have 
    \begin{equation*}
        h^2 \ll N\log_q \frac{N}{h}.
    \end{equation*}

    Now, $N \log_q \frac{N}{h}$ is a decreasing function of $h$, so for $h \ll \sqrt{N}$ we have: 
    \begin{equation*}
    \begin{split}
        N \log_q \frac{N}{h} &\gg N \log_q \sqrt{N}\\
        &\gg h^2,
    \end{split}
    \end{equation*}
    as required.
\end{proof}

\begin{proof} [Proof of Proposition \ref{SmoothPF}]
We define the sequence $\{a_{G}\}_{G \in \mathcal{M}_{n}} \subset \mathbb{C}$ to be the indicator function of an $h$-smooth polynomial.
\begin{equation*}
    a_{G} = \begin{cases}
        \lambda(G), &G \text{ is $h$-smooth}\\
        0, &\text{otherwise.} \\
    \end{cases}
\end{equation*}

Now we have 
\begin{equation*}
    \sum_{\substack{{\chi \Mod{t^{N-h}}}\\{\chi \text{ even}}}}|\sum_{\substack{{G \in \mathcal{M}_{n}}\\{G \in \mathcal{S}_{h,n}}}} \lambda(G)\chi(G)|^{2} = \sum_{\substack{{\chi \Mod{t^{N-h}}}\\{\chi \text{ even}}}}|\sum_{G \in \mathcal{M}_{n}}a_G \chi(G)|^{2}.
\end{equation*}

Furthermore, by Lemma \ref{bound no smooth} we know $\sum_{G \in \mathcal{M}_n}|a_G| = |\mathcal{S}_{h,n}|$ is small for $h$ in the given range. Thus using Lemma \ref{MVT} we have

\begin{equation*}
    \begin{split}
        &\sum_{\substack{{\chi \Mod{t^{N-h}}}\\{\chi \text{ even}}}}|\sum_{\substack{{G \in \mathcal{M}_{n}}\\{G \in \mathcal{S}_{h,n}}}}\lambda(G)\chi(G)|^{2} \ll q^{N-h} (q^{n-(N-h)} + 1)\sum_{G \in \mathcal{M}_n}|a_G|\\
        &\ll (q^n + q^{N-h})|\mathcal{S}_{h,n}| \ll (q^n + q^{N-h})|\mathcal{S}_{h,N}| \ll q^{N+n-h} + q^{2(N-h)}.\\
    \end{split}
\end{equation*}
\end{proof}

\pagebreak

\subsection*{Collecting Terms}
\begin{proof}[Proof of Theorem \ref{Result}]
    Substituting Proposition \ref{LargePF} and Proposition \ref{SmoothPF} into Equation \ref{split fml} we get:

\begin{equation*}
\begin{split}
    \Var(\lambda_{N,h}) 
    &\ll \frac{1}{q^{2(N-h)}}\left( \sum_{n=0}^{N}\left(\frac{N^3}{h^2}q^{N+n-h} + q^{2(N-h)}\right)^{1/2}\right)^2\\
    &\ll \frac{1}{q^{2(N-h)}}\left( N\left(\frac{N^3}{h^2}q^{2N-h} + q^{2(N-h)}\right)^{1/2}\right)^2\\
    &= \frac{1}{q^{2(N-h)}} (\frac{N^5}{h^2}q^{2N-h}+N^2 q^{2(N-h)})\\
    &\ll \frac{N^5}{h^2}q^h.
\end{split}
\end{equation*}

As required.
\end{proof}

\pagebreak
\bibliographystyle{plain} 
\bibliography{references}

\begin{thebibliography}{10}

\bibitem{apostol2013introduction}
Tom~M Apostol.
\newblock {\em Introduction to analytic number theory}.
\newblock Springer Science \& Business Media, 2013.

\bibitem{CHsuml}
Jake Chinis.
\newblock On the liouville function in short intervals.
\newblock {\em International Mathematics Research Notices}, 2022(15):11203--11219, 2022.

\bibitem{Nohsmooth}
Theodoulos Garefalakis and Daniel Panario.
\newblock Polynomials over finite fields free from large and small degree irreducible factors.
\newblock {\em Journal of Algorithms}, 44(1):98--120, 2002.

\bibitem{ANT}
Henryk Iwaniec and Emmanuel Kowalski.
\newblock {\em Analytic number theory}, volume~53.
\newblock American Mathematical Soc., 2021.

\bibitem{Katzequi}
Nicholas~M Katz.
\newblock a question of keating and rudnick.
\newblock {\em Int. Math. Res. Notices}, 16:3613--3638, 2014.

\bibitem{KRvar}
Jonathan Keating and Zeev Rudnick.
\newblock Squarefree polynomials and m{\"o}bius values in short intervals and arithmetic progressions.
\newblock {\em Algebra \& Number Theory}, 10(2):375--420, 2016.

\bibitem{corrff}
Oleksiy Klurman, Alexander~P Mangerel, and Joni Ter{\"a}v{\"a}inen.
\newblock Correlations of multiplicative functions in function fields.
\newblock {\em Mathematika}, 69(1):155--231, 2023.

\bibitem{klurman2023multiplicative}
Oleksiy Klurman, Alexander~P Mangerel, and Joni Ter{\"a}v{\"a}inen.
\newblock Multiplicative functions in short arithmetic progressions.
\newblock {\em Proceedings of the London Mathematical Society}, 127(2):366--446, 2023.

\bibitem{MRthm}
Kaisa Matom{\"a}ki and Maksym Radziwi{\l}{\l}.
\newblock Multiplicative functions in short intervals.
\newblock {\em Annals of Mathematics}, pages 1015--1056, 2016.

\bibitem{GeorgesRhin1972}
Georges Rhin.
\newblock {\em R{\'e}partition modulo 1 dans un corps de s{\'e}ries formelles sur un corps fini}.
\newblock 1972.

\bibitem{stein1970singular}
Elias~M Stein.
\newblock {\em Singular integrals and differentiability properties of functions}.
\newblock Princeton university press, 1970.

\bibitem{TRZF}
Edward~Charles Titchmarsh and David~Rodney Heath-Brown.
\newblock {\em The theory of the Riemann zeta-function}.
\newblock Oxford university press, 1986.

\end{thebibliography}

\end{document}